\documentclass[11pt,reqno]{amsart}
\usepackage[body={7in,9in},centering]{geometry}
\usepackage{amsmath,amssymb,amsthm,mathrsfs}
\usepackage{color,xcolor}
\definecolor{cobalt}{RGB}{61,99,181}
\usepackage[colorlinks,citecolor=cobalt,linkcolor=cobalt]{hyperref}

\newtheorem{thm}{Theorem}[section]
\newtheorem{defi}[thm]{Definition}

\newtheorem{lem}[thm]{Lemma}

\newtheorem{prop}[thm]{Proposition}

\newtheorem{ques}[thm]{Question}

\numberwithin{equation}{section}

\date{\today}

\makeatletter

\newcommand{\Rmnum}[1]{\expandafter\@slowromancap\romannumeral #1@}
\makeatother

\begin{document}

\title[The Douglas question]{The Douglas question on the Bergman and Fock spaces}

\author[Chen]{Jian-hua Chen}
\address{School of  Mathematical  and  Computational Science, Hunan  University  of  Science  and  Technology,
Xiangtan, Hunan, 411201, P. R. China}
\email{cjhmath@mail.ustc.edu.cn}

\author[Leng]{Qianrui Leng}
\address{College of Mathematics and Statistics, Chongqing University, Chongqing, 401331, P. R. China}
\email{17788619117@163.com}

\author[Zhao]{Xianfeng Zhao}
\address{College of Mathematics and Statistics, Chongqing University, Chongqing, 401331, P. R. China}
\email{xianfengzhao@cqu.edu.cn}



\begin{abstract}
Let $\mu$ be a  positive Borel measure and $T_\mu$ be the bounded Toeplitz operator induced by $\mu$ on the Bergman or Fock space. In this paper, we mainly investigate the invertibility of the Toeplitz operator $T_\mu$ and the Douglas question on the Bergman and Fock spaces.
In the  Bergman-space setting, we obtain several necessary and sufficient conditions for the invertibility of $T_\mu$ in terms of  the Berezin transform of $\mu$ and the reverse Carleson condition in two classical cases:
(1) $\mu$ is absolutely continuous with respect to the normalized area measure on the open unit disk $\mathbb D$;
(2) $\mu$  is the pull-back measure of the normalized area measure  under an analytic self-mapping  of $\mathbb D$.
Nonetheless, we show that there exists a Carleson measure for the Bergman space such that its Berezin transform is bounded below but the corresponding Toeplitz operator is not invertible. On the Fock space, we show that $T_\mu$ is invertible if and only if $\mu$ is a reverse Carleson measure, but the invertibility of $T_\mu$ is not completely determined by the invertibility of the Berezin transform of $\mu$. These suggest that the answers to the Douglas question for Toeplitz operators induced by positive measures on the Bergman and Fock spaces are both negative in general cases.
\end{abstract} \maketitle

\tableofcontents

\section{Introduction}\label{In}

Let us  begin with certain notation and basic knowledge about the  Bergman
space, the Fock space and Toeplitz operators on them.  Let $dA(z)=\frac{1}{\pi}dxdy$ denote the normalized Lebesgue area measure on the complex plane $\mathbb{C}$. $L^2(\mathbb D, dA)$ denotes the
space of the Lebesgue measurable functions on the open unit disk $\mathbb D=\{z\in \mathbb C: |z|<1\}$ with the following norm:
$$\|f\|_{L^2 (\mathbb{D}, d A)}=\bigg[\int_{\mathbb D} |f(z)|^2 dA(z)\bigg]^{\frac{1}{2}}<+\infty.$$
The Bergman space $L_a^2$ is the closed subspace of $L^2(\mathbb D, dA)$ consisting of analytic functions on $\mathbb D$, which is a Hilbert space with the inner product
$$\langle f, g\rangle_{L_a^2}=\int_{\mathbb D} f(z)\overline{g(z)}dA(z).$$
Recall that $L_a^2$  is a reproducing kernel Hilbert space and the reproducing kernel at $z\in \mathbb D$ is given by
$$K_z(w)=\frac{1}{(1-\overline{z}w)^2},  \ \ \ \  w\in \mathbb D.$$
The unit vector  $k_z=\frac{K_z}{\|K_z\|_{L_a^2}}$ is called the normalized reproducing kernel for $L_a^2$.

Letting $P$ be the orthogonal projection from $L^2(\mathbb D, dA)$ onto $L_a^2$, then the integral representation of $P$ is given by
$$Pf(z)=\int_{\mathbb D}f(w)\overline{K_{z}(w)}dA(w)=\int_{\mathbb D}\frac{f(w)}{(1-z\overline{w})^2}dA(w), \ \ \ \ f\in L^2(\mathbb D, dA).$$
Given a function $\varphi\in L^\infty(\mathbb D, dA)$, the Toeplitz operator $T_\varphi$ with symbol $\varphi$ on the Bergman space  $L_a^2$ is defined by
$$T_\varphi f=P(\varphi f), \ \ \ \ f\in L_a^2.$$
Using the integral representation of $P$, we can write $T_\varphi$ as an integral operator on the Bergman space:
$$T_{\varphi} f(z)=\int_{\mathbb D}\frac{\varphi(w)f(w)}{(1-z\overline{w})^2}dA(w),$$
where $f\in L_a^2$ and $z\in \mathbb D$.

Toeplitz operators can also be defined for finite Borel measures on $\mathbb D$. In fact, if $\mu$ is a finite (complex) Borel measure on the open unit disk $\mathbb D$, then the Toeplitz operator induced by the measure $\mu$ is defined by
$$T_\mu f(z)=\int_{\mathbb D}\frac{f(w)}{(1-z\overline{w})^2}d\mu(w), \ \ \ f\in L_a^2.$$
In this case,  the Berezin transform of $\mu$ (or $T_\mu$) is the function defined by
 $$\widetilde{\mu}(z)=\int_{\mathbb D}|k_z(w)|^2d\mu(w),\ \ \ \ z\in \mathbb D.$$
If $d\mu=\varphi dA$ for a function  $\varphi\in L^1(\mathbb D, dA)$, then we simply write $T_{\mu}=T_{\varphi}$ and
$$\widetilde{\mu}(z)=\int_{\mathbb D}\varphi(w)|k_z(w)|^2dA(w)$$
is usually denoted by $\widetilde{\varphi}$, which is also called the Berezin transform of $\varphi$.
In particular, when $\varphi \in L^ {\infty} (\mathbb{D}, dA)$, we return to the classical Toeplitz operator defined in the preceding paragraph.

We now turn to the introduction of the Fock space and Toeplitz operators on such space.  Let $$d\lambda(z)=\frac{1}{2}\mathrm{e}^{-\frac{|z|^2}{2}}dA(z)$$  be the Gaussian measure on the complex plane $\mathbb C$. The Fock space $\mathcal F^2$ consists of all entire functions $f$ in $L^2(\mathbb C, d\lambda)$. One can show easily that $\mathcal F^2$ is a closed subspace of $L^2(\mathbb C, d\lambda)$. Consequently, the Fock space $\mathcal F^2$ is a Hilbert space with the inner product:
$$\langle f, g\rangle_{\mathcal F^2}=\int_{\mathbb C}f(z)\overline{g(z)}d\lambda(z).$$
Moreover, $\mathcal F^2$  is also  a reproducing kernel Hilbert space and the reproducing kernel at $z\in \mathbb C$ is given by
$$K_z^{(\mathcal F^2)}(w)=\mathrm{e}^{\frac{\overline{z}w}{2}},  \ \ \ \  w\in \mathbb C.$$

Similar to the  Bergman space, the orthogonal projection $P^ {(\mathcal{F}^ 2)}$ from $L^ 2 (\mathbb{C}, d \lambda)$ onto $\mathcal{F}^ 2$ is the following  integral operator:
\begin{align*}
	P^ {(\mathcal{F}^ 2)} f (z) = \int_{\mathbb{C}} f(w) \overline{K_z^ {(\mathcal{F}^ 2)} (w)} d \lambda(w) = \int_{\mathbb{C}} f(w) \mathrm{e}^ {\frac{z \overline{w}}{2}} d A (w), \ \ \ \ f \in L^ 2 (\mathbb{C}, d \lambda).
\end{align*}
Given a function $\varphi \in L^ {\infty} (\mathbb{C}, dA)$, the Toeplitz operator $T_{\varphi}$ with symbol $\varphi$ on the Fock space is defined by
\begin{align*}
	T_{\varphi} f = P^ {(\mathcal{F}^ 2)} (\varphi f), \ \ \ \ f \in \mathcal{F}^ 2,
\end{align*}
which can be written as
\begin{align*}
	T_{\varphi} f (z) = \int_{\mathbb{C}} f(w) \varphi (w) \mathrm{e}^ {\frac{z \overline{w}}{2}} d \lambda (w),
\end{align*}
where $f \in \mathcal{F}^ 2$ and $z \in \mathbb{C}$.

Let $\mu$ be a complex Borel measure on $\mathbb C$ satisfying the\emph{ condition (M)}, that is,
\begin{align}
\int_{\mathbb C} \big|K_z^{(\mathcal F^2)}(w)\big|^2e^{-\frac{|w|^2}{2}}d|\mu|(w)<+\infty
\end{align}
for all $z\in \mathbb C$, where $|\mu|$ denotes the total variation of $\mu$. Under this assumption, the Toeplitz operator induced by  $\mu$  can be defined on a dense subspace of $\mathcal{F}^ 2$:
$$T_\mu f(z)=\frac{1}{2}\int_{\mathbb C} f(w)\overline{K_z^{(\mathcal F^2)}(w)}\mathrm{e}^{-\frac{|w|^2}{2}}d\mu(w), \ \ \ \  z\in \mathbb C.$$
Condition (M) also guarantees that the Berezin transform of $\mu$ is well-defined:
$$\widetilde{\mu}(z)=\frac{1}{2}\int_{\mathbb C} |k_z^{(\mathcal F^2)}(w)|^2 \mathrm{e}^ {-\frac{|w|^ 2}{2}} d \mu(w) = \frac{1}{2} \int_{\mathbb{C}} \mathrm{e}^ {-\frac{|z - w|^ 2}{2}} d \mu (w), $$
where
$$k_z^{(\mathcal F^2)}(w)=\frac{K_z^{(\mathcal F^2)}(w)}{\|K_z^{(\mathcal F^2)}\|_{\mathcal F^2}}=\mathrm{e}^{\frac{w\overline{z}}{2}-\frac{|z|^2}{4}} \ \  \ \  \  (z, w\in\mathbb C)$$
is the normalized reproducing kernel for $\mathcal F^2$.  When $d\mu=\varphi dA$ with $\varphi \in L^ {\infty} (\mathbb{C}, d \lambda)$, we get back to $T_\varphi$, and in this case $\widetilde{\mu}$ is again denoted by $\widetilde{\varphi}$, which is also called the Berezin transform of $\varphi$. Notice that the dual use of the notation ``$T_\varphi$" and ``$T_{\mu}$'' for Toeplitz operators on the Bergman
and Fock spaces should cause no confusion.

For more knowledge about the Bergman and Fock spaces, the Berezin transform and Toeplitz operators on such spaces,  one can consult for instance \cite{Ax}, \cite{Zhu}, \cite{AZ} and \cite{ZZ1}.
Usually, Toeplitz operators on the Bergman space  are called  ``Bergman-Toeplitz operators" and  Toeplitz operators on the Fock space  are called  ``Fock-Toeplitz operators" for short.

Toeplitz operators induced by measures were investigated by many authors in the past few decades. However, Luecking \cite{Lue1} is the first to consider Toeplitz operators on the Bergman space with measures as symbols. Later, using the Berezin transform and \emph{Carleson measures} (which will be introduced in the next section), Zhu characterized the boundedness and compactness of Toeplitz operators induced by positive measures on  weighted Bergman spaces of bounded symmetric domains, see \cite{Zhu2} and \cite{Zhu}. The characterization of Schatten class Toeplitz operators with measures as symbols on the  Bergman space of the unit ball was obtained by Zhu \cite{Zhu3, Zhu}. In 2008, Luecking \cite{Lue2} studied the finite rank Toeplitz operator induced by a complex  measure on the Bergman space.

An especially important but quite difficult problem in the Toeplitz operator theory is to determine the invertibility of Toeplitz operators in terms of the properties of their symbols. The original Douglas question is an influential open question concerning the invertibility problem of Hardy-Toeplitz operators, which says that: \emph{if the harmonic extension (poisson integral) of $\varphi$ is bounded below on the unit disk $\mathbb D$, then is the Toeplitz operator with symbol $\varphi$  invertible on the Hardy  space?} The original Douglas question was studied by many authors, for instance, Tolokonnikov \cite{Tol}, Nikolskii \cite{Nik}, Wolff \cite{Wol} and  Karaev \cite{Kar}. Observe that the harmonic extension of a bounded function $\varphi$ defined on $\partial \mathbb D$ is equal to the Berezin transform of the Toeplitz operator with symbol $\varphi$ on the Hardy space.
This leads to the following natural question for Bergman-Toeplitz operators analogous to the Douglas question on the Hardy space:

\begin{ques}\label{Douglas}
Let $\varphi$ be in $L^\infty (\mathbb D, dA)$. Then is $T_{\varphi}$ invertible  on the Bergman space $L_a^2$ if the Berezin transform $|\widetilde{\varphi}(z)| \geqslant \delta$ for some positive constant
 $\delta$
 and for all $z\in \mathbb D$?
 \end{ques}

The above question is answered positively  for Bergman-Toeplitz operators with  bounded (conjugate) analytic  symbols,
and real harmonic symbols in \cite{MS}. In \cite{Lue}, by constructing a reverse Carleson measure for the Bergman space $L_a^2$, Luecking established a number of necessary and sufficient conditions for Toeplitz operators with bounded positive symbols to be invertible on the Bergman space. Based on Luecking's work, Zhao and Zheng  showed that the answer to Question \ref{Douglas} is also affirmative for Bergman-Toeplitz operators with bounded positive symbols \cite{ZZ}. \v{C}u\v{c}kovi\'{c} and Vasaturo \cite{Cu} obtained a necessary and sufficient condition for the invertibility of Toeplitz operators whose symbols are averaging functions of  Carleson measures for the (standard) weighted Bergman space.
Recently, Guo, Zhao and Zheng \cite{ZZ2, GZZ}
provided some partial answers to this question for Bergman-Toeplitz operators with some harmonic polynomials. However, Question \ref{Douglas} remains open for general harmonic symbols.

%
%

In the present  paper, we are interested in the invertibility problem of Bergman-Toeplitz  and Fock-Toeplitz operators  with positive measures as  symbols.  In particular, we will focus on the Douglas question for Toeplitz operators induced by positive measures on the Bergman and Fock spaces:
\begin{ques}\label{Q2}
 Suppose that $\mu$ is a finite  positive Borel measure on the unit disk $\mathbb D$ and $T_{\mu}$ is bounded on the Bergman space $L_a^2$. Then is $T_\mu$ invertible on $L_a^2$ if the Berezin transform of $\mu$ is bounded below on $\mathbb D$?
\end{ques}

\begin{ques}\label{Q3}
 Suppose that $\mu$ is a  positive Borel measure on the complex plane $\mathbb C$ and $T_{\mu}$ is bounded on the Fock space $\mathcal F^2$. Then is $T_\mu$ invertible on $\mathcal F^2$ if the Berezin transform of $\mu$ is bounded below on $\mathbb C$?
\end{ques}

As mentioned above, the answer to Question \ref{Q2} is positive in the case that $d\mu=\varphi dA$, where $\varphi$ is a nonnegative function in $L^\infty(\mathbb D, dA)$.  Apart from this,  very little is known about the invertibility of Toeplitz operators induced by measures, even if the measure is positive. Using the Berezin transform and the reverse Carleson condition,  in Section \ref{DB} we establish a number of necessary and sufficient conditions for Toeplitz operators with positive measures as symbols to be invertible on the Bergman space in two classical cases. More precisely, in the case that $\mu$  is absolutely continuous with respect to the Lebesgue measure $dA$ and the Radon-Nikodym derivative of $d\mu$ with respect to $dA$ is bounded, we show in Theorem \ref{M1} that $T_\mu$ is invertible if and only if the Berezin transform of  $\mu$ is invertible in $L^\infty(\mathbb D, dA)$, if and only if $\mu$ satisfies the \emph{reverse Carleson condition for $L_a^2$} (see (\ref{RCC}) for the definition). For the case that $\mu$ is the pull-back measure of $dA$ under an analytic self-mapping of $\mathbb D$, we show in Theorem \ref{M2} that the invertibility of $T_\mu$ and the Berezin transform of  $\mu$ are equivalent. Note that the Berezin transform of $\mu$ is bounded below on $\mathbb D$ provided that $T_\mu$ is bounded below on the Bergman space.
 Combining this with the conclusions in Theorems \ref{M1} and \ref{M2}, one may conjecture that the answer to Question \ref{Q2} is affirmative.  Surprisingly,  we will show that there exists a finite positive Borel measure on $\mathbb D$ such that the Berezin transform of $\mu$ is invertible in $L^\infty(\mathbb D, dA)$, but the corresponding Toeplitz operator $T_\mu$ is not invertible on the Bergman space (see Theorem \ref{CE}).  This implies that the answer to Question \ref{Q2} is negative in general.

In the setting of the Fock space, Question \ref{Q3} was answered positively by Wang and Zhao \cite{WZ} when $d\mu=\varphi dA$, where $\varphi$ is a bounded positive function on the complex  plane. However,  the answer to this question is unknown for general positive measures.
In the final section, we investigate the invertibility  of Fock-Toeplitz operators induced by positive measures  and discuss the analogues of Theorems \ref{M1} and \ref{CE} in the Fock-space setting. Specifically, we establish in Theorem \ref{MF} that the Fock-Toeplitz operator  $T_\mu$  is invertible if and only if $\mu$ is a reverse Carleson measure for the Fock space, where $\mu$ is a positive Borel measure on the complex plane $\mathbb C$. Based on this result, in Theorem \ref{FM1} we establish several equivalent characterizations for $T_\mu$ to be invertible on the Fock space via the Berezin transform and the \emph{reverse Carleson condition for $\mathcal F^2$} (see (\ref{RCCF}) for the definition).  In addition, we  show in Theorem  \ref{TF} that there exists a positive Borel  measure $\nu$ on $\mathbb C$ such that $T_\nu$ is not invertible on the Fock space but the Berezin transform of $\nu$ is  bounded below on $\mathbb C$. Therefore, the answer to the Douglas question for Toeplitz operators induced by positive measures on the Fock space is also negative.

\section{Preliminaries}\label{2}

Before studying the invertibility of Toeplitz operators induced by positive measures on the Bergman and Fock spaces, we need to fix some notation and review certain necessary preliminary knowledge first.

\subsection{Necessary knowledge for the Bergman space}
\
\newline
\indent

For $w\in \mathbb D$, let $\varphi_w$  be the M\"{o}bius transform defined by
$$\varphi_w(z)=\frac{w-z}{1-\overline{w}z},  \ \ \  \ z\in \mathbb D.$$
Note that the function $\varphi_w$ is a one-to-one analytic mapping of $\mathbb D$ onto $\mathbb D$.  Let $w$ and $z$ be two points in  $\mathbb D$. The pseudo-hyperbolic distance $\rho(w, z)$ and the Bergman distance $\beta (w, z)$ between $w$ and $z$ are defined by
$$\rho(z, w)=|\varphi_w(z)| \ \ \ \ \  \textrm{and} \ \ \ \ \ \beta (z, w) = \tanh^ {-1} {\rho (z, w)},$$
respectively.  An application of  Schwarz's lemma shows that $\rho$ is a metric on the unit disk $\mathbb D$, see \cite[Lemma 1.4]{Gar} if necessary. Thus, $\beta$ is also a metric on $\mathbb D$.
In addition, the metrics $\rho$ and $\beta$ are both invariant under the action of  M\"{o}bius transforms.

Let $w\in \mathbb C$ and $r > 0$.  In the following, we use $B(w, r)$  to denote the  Euclidean disk with center $w$ and radius $r$:
$$B(w, r)=\{z\in \mathbb C: |w - z| < r\}.$$
The pseudo-hyperbolic disk $E(w, r)$ with center $w \in \mathbb D$ and radius $r \in (0, 1)$ is defined by
$$E(w, r)=\{z\in \mathbb D: \rho(w, z)<r\}.$$
Moreover, the Bergman disk $D(w, r)$ with center $w \in \mathbb D$ and radius $r$ is defined by
$$D(w, r)=\{z\in \mathbb D: \beta(w, z)<r\}. $$

From \cite[Proposition 4.4]{Zhu}, the Lebesgue measure of $E(z, r)$ is comparable to the Lebesgue measure of $E(w, r)$ when $\rho(z, w)$ stays bounded:
\begin{align}\label{zw}
  A(E(z, r)) \asymp A(E(w, r))
  \end{align}
 when $\rho (z, w) < c$, where $c \in (0, 1)$ is a constant. Here and in the following, the notation $A \lesssim B$ for two nonnegative quantities $A$ and $B$ means that there is an inessential  positive constant $C_1$ such that $A\leqslant C_1B$. Similarly, the notation $A \gtrsim  B$  stands for that there is an inessential  positive constant $C_2$ such that $A\geqslant C_2B$. By $A \asymp B$ for two nonnegative quantities $A$ and $B$, we mean that $A \lesssim B$ and $A \gtrsim  B$.

Let $\mu$ be a finite positive Borel measure on $\mathbb D$. We say that $\mu$ is a \emph{Carleson measure for} $L_a^2$ if there exists an absolute  constant $C>0$ such that
$$\int_{\mathbb D}|f(z)|^2d\mu(z)\leqslant C \int_{\mathbb D}|f(z)|^2dA(z)$$
for all $f\in L_a^2$; We say  that $\mu$ is a \emph{reverse Carleson measure for} $L_a^ 2$ if there exists an absolute constant $C > 0$ such that
$$\int_{\mathbb D}|f(z)|^2d\mu(z)\geqslant C \int_{\mathbb D}|f(z)|^2dA(z)$$
for all $f \in L_a^ 2$. Moreover, we call that $\mu$ satisfies the \emph{reverse Carleson condition} if there is  an $r \in (0, 1)$ and a positive constant $C$ so that
\begin{align} \label{RCC}
	C^ {-1} \leqslant \frac{\mu (E(z, r))}{A(E(z, r))} \leqslant C
\end{align}
for all $z \in \mathbb{D}$.

It is well-known that $\mu$ is a Carleson measure for $L_a^2$ if and only if for every radius $r \in (0, 1)$, we can find a positive constant $C$ depending only on $r$ such that
 $$\frac{\mu(E(z, r))}{A(E(z, r))}\leqslant C$$
 for all $z\in \mathbb D$, see \cite{Zhu} if needed. Using the reverse Carleson condition for $L_a^2$, Ghatage and Tjani \cite{GT} provided some necessary and sufficient conditions for a composition operator to have closed range on the Bergman space.


Let  $G$ be a measurable subset of $\mathbb D$ and $\chi_{G}$ be the characteristic function of $G$. It was established by Luecking \cite{Lue} that $\chi_G dA$ is a reverse Carleson measure for $L_a^2$ if and only if there is  a $\delta>0$ and an $\eta\in (0, 1)$ such that
$$A(G\cap E(a, \eta))>\delta A(E(a, \eta))$$
for all $a\in \mathbb D$,  which is equivalent to that the measure $\chi_G dA$ satisfies the reverse Carleson condition for $L_a^2$.
Moreover, Ghatage and Tjani  \cite{GT} showed that the Carleson measure for the Bergman space has the following important property:
\begin{lem}\label{Lim}
\emph{(\cite[Proposition 4.1]{GT}) }If $\mu$ is  a  Carleson measure for the Bergman space, then
$$\lim_{r\rightarrow 1^-}\sup_{z\in \mathbb D} \int_{\mathbb D\backslash E(z, r)}\frac{(1-|z|^2)^2}{|1-\overline{z}w|^4}d\mu(w)=0.$$
\end{lem}

We end this subsection by introducing the definitions of separated, sampling and  interpolation sequences for the Bergman space, which will be needed later on.  For more information concerning these concepts, one can consult \cite{DS}, \cite{Hor} and \cite{Sei}.

\begin{defi}\label{def}
Let $\{ u_k \}_{k = 1}^ {\infty}$ be a sequence in the open unit disk  $\mathbb D$.
\begin{itemize}
\item  $\{u_k \}_{k = 1}^ {\infty}$ is called \textit{uniformly discrete} or \textit{separated} if there is a constant $\varepsilon \in (0, 1)$ such that $$\rho(u_k, u_j) >\varepsilon \ \ \  \text{for}  \ \ \  k \neq j.$$
In this case, the separation constant of $\{ u_k \}_{k = 1}^ {\infty}$ is defined by
$$ \delta (\{ u_k \}_{k = 1}^ {\infty}) = \inf_{k \neq j} {\rho (a_k, a_j)} .$$
\item We call $\{ u_k \}_{k = 1}^ {\infty}$ an $r$-net if
$$ \mathbb D = \bigcup_{k = 1}^ {\infty} {E(u_k, r)} $$
for some $r \in (0, 1)$.
\item  If $u_k \neq u_j$ for $k \neq j$ and
\begin{align*}
	\int_{\mathbb D} |f(z)|^ 2 d A(z) \asymp \sum_{k = 1}^ {\infty} {(1 - |u_k|^ 2)^ 2 |f(u_k)|^ 2}, \ \ \ \ f \in L_a^ 2,
\end{align*}
then  $\{ u_k \}_{k = 1}^ {\infty}$ is  called a \textit{sampling sequence for} $L_a^2$.\\
\item If $u_k \neq u_j$ for $k \neq j$ and the interpolation problem $f(u_k) = w_k$ \emph{(}$k=1, 2,\cdots$\emph{)} has a solution $f \in L_a^ 2$ whenever
\begin{align*}
	\sum_{k = 1}^ {\infty} {(1 - |u_k|^ 2)^ 2 |w_k|^ 2} < +\infty,
\end{align*}
then we say that $\{ u_k \}_{k = 1}^ {\infty}$  is an \textit{interpolation sequence} for $L_a^2$.
\end{itemize}
\end{defi}

There is an important class of uniformly discrete nets which is called the lattice. A lattice for $L_a^ 2$ is a sequence $\{ u_k \}_{k = 1}^ {\infty}$ of distinct points in $\mathbb{D}$ such that
\begin{align*}
	\mathbb{D} = \bigcup_{k = 1}^ {\infty} {D(u_k, R)} = \bigcup_{k = 1}^ {\infty} {E(u_k, r)}
\end{align*}
and $\beta (u_k, u_j) \geqslant  \frac{R}{2}$ for $k \neq j$, where $R > 0$ and $r = \tanh{(R)}$. Clearly, $\{ u_k \}_{k = 1}^ {\infty}$ is an $r$-net with the separation constant greater than or equal to $\tanh{(\frac{R}{2})}$.

\subsection{Necessary knowledge for the  Fock space}
\
\newline
\indent

Let $\mu$ be a  positive Borel measure $\mu$ on the complex plane  $\mathbb{C}$. Then $\mu$ is called a \emph{Carleson measure for the Fock space $\mathcal{F}^ 2$} if there exists a constant $C>0$ such that
	\begin{align*}
		\int_{\mathbb{C}} |f(w)|^ 2 \mathrm{e}^ {-\frac{|w|^ 2}{2}} d \mu (w) \leqslant  C \int_{\mathbb{C}} |f(w)|^ 2 \mathrm{e}^ {-\frac{|w|^ 2}{2}} d A(w),  \ \  \   \ \ f \in \mathcal{F}^ 2.
	\end{align*}
On the other hand,  $\mu$ is called a \emph{reverse Carleson measure for $\mathcal{F}^ 2$} if there exists a constant $C > 0$ such that
	\begin{align*}
	\int_{\mathbb{C}} |f(w)|^ 2 \mathrm{e}^ {-\frac{|w|^ 2}{2}} d \mu (w)\geqslant  C\int_{\mathbb{C}} |f(w)|^ 2 \mathrm{e}^ {-\frac{|w|^ 2}{2}} d A(w), \ \ \ \ \ f \in \mathcal{F}^ 2.
	\end{align*}

 Like in the case of the Bergman space, we say that a positive Borel measure $\mu$ on $\mathbb{C}$ satisfies  the\emph{ reverse Carleson condition for $\mathcal F^2$} if there exists an $r>0$ and a positive constant $C$ such that
	\begin{align}\label{RCCF}
		C^{-1} \leqslant \mu (B(z, r)) \leqslant  C, \ \ \ \ \  z \in \mathbb{C}.
	\end{align}

There are several concepts for the Fock space  parallel to those in Definition \ref{def}, which can be found in Chapter 4 of the book \cite{Zhu1}.
\begin{defi}
	Let $\{ u_k \}_{k = 1}^ {\infty}$ be a sequence in the complex plane  $\mathbb C$.
	\begin{itemize}
		\item  $\{u_k \}_{k = 1}^ {\infty}$ is called \textit{uniformly discrete} or \textit{separated} if there is a constant $\varepsilon > 0$ such that $$|u_k - u_j| >\varepsilon \ \ \  \text{for}  \ \ \  k \neq j.$$
		\item We call $\{ u_k \}_{k = 1}^ {\infty}$ an $r$-net if
		$$ \mathbb C = \bigcup_{k = 1}^ {\infty} {B(u_k, r)} $$
		for some $r > 0$.
		\item  If $u_k \neq u_j$ for $k \neq j$ and
		\begin{align*}
			\int_{\mathbb D} |f(z)|^ 2 d \lambda(z) \asymp \sum_{k = 1}^ {\infty} {|f(u_k)|^ 2 \mathrm{e}^ {-\frac{|u_k|^ 2}{2}}}
		\end{align*}
	for $f \in \mathcal{F}^2$,	then  $\{ u_k \}_{k = 1}^ {\infty}$ is  called a \textit{sampling sequence} for the Fock space $\mathcal F^2$.
	\end{itemize}
\end{defi}

In the setting of the Fock space, we can also define a lattice. Let $\omega$ be any complex number and $\omega_1$ and $\omega_2$ be any two nonzero complex numbers such that their ratio is not real. The set
$$\Lambda(\omega, \omega_1, \omega_2)=\{\omega+m\omega_1+n\omega_2: m, n\in \mathbb Z\}$$
is called the lattice for $\mathcal F^ 2$ generated by $\omega$, $\omega_1$ and $\omega_2$, where $\mathbb Z$ denotes the set of integers.  In the following, we only consider the case when $\omega=0$, $\omega_1=r$ and $\omega_2=\mathrm{i}r$ with $r>0$. In this case, we denote
\begin{align}\label{Lambda}
	\Lambda(0, r, \mathrm{i}r)=\{a_k\}_{k=1}^\infty.
\end{align}
Clearly, $\Lambda(0, r, \mathrm{i}r)$ is a separated $2 r$-net for $\mathcal F^2$.

%


\section{The Douglas question on the Bergman space}\label{DB}

In Section \ref{DB}, we study the invertibility of Bergman-Toeplitz operators with positive measures as symbols and the Douglas question
for such class of Toeplitz operators.  This section is organized as follows. Firstly, we discuss two classic kinds of Bergman-Toeplitz operators:
\begin{itemize}
\item  [] \textbf{Case I.}  $d\mu=\varphi dA$, where $\varphi$ is a nonnegative function in $L^\infty(\mathbb D, dA)$;\\
\item  [] \textbf{Case II.}  $\mu$ is the pull-back measure of $dA$ under an analytic self-mapping of $\mathbb D$.
\end{itemize}
Both of the above two cases provide a positive answer to the Douglas question, i.e., the invertibility of $T_{\mu}$ and $\widetilde{\mu}$ are equivalent.
Secondly, using Lemmas \ref{M} and \ref{CRC} (which give equivalent characterizations for  the invertibility of $T_{\mu}$ and $\widetilde{\mu}$) we transform the study of the Douglas question into the discussion of  properties of $\mu$. Then by constructing a special kind of measures, we can  answer Question \ref{Q2} negatively.

Let us begin with the following three lemmas, which provide characterizations for the boundedness and invertibility of a Toeplitz operator and its Berezin transform. The latter two are particularly important for the investigation  of the Douglas question on the Bergman space.

\begin{lem}\label{BC}
\emph{(\cite[Theorem 7.5]{Zhu})}  Let $\mu$ be a finite positive Borel measure on the unit disk $\mathbb D$. Then the following are equivalent:\\
$(a)$ $T_\mu$ is bounded on the Bergman space $L_a^2$;\\
$(b)$ $\widetilde{\mu}$ is bounded on $\mathbb D$;\\
$(c)$ $\mu$ is a Carleson measure for $L_a^2$.
\end{lem}

\begin{lem}\label{M}
\emph{(\cite[Proposition 2]{Cu})}   Let $\mu$ be a finite positive Borel measure on $\mathbb D$. If the Toeplitz operator $T_{\mu}$ is bounded on the Bergman space $L_a^2$, then the following are equivalent:\\
$(a)$ $T_\mu$  is invertible on $L_a^2$;\\
$(b)$ $\mu$ is a reverse Carleson measure for $L_a^2$.
\end{lem}

\begin{lem}\label{CRC}
\emph{(\cite[Theorem 4.1]{GT})} A Carleson measure $\mu$  for the Bergman space satisfies the  reverse Carleson condition if and only if the Berezin transform $\widetilde{\mu}$ is bounded above and below on $\mathbb D$.
\end{lem}

Applying Lemma \ref{CRC} and combining the conclusions in \cite{Lue, ZZ} and Lemma \ref{M}, we obtain the following five equivalent characterizations for the invertibility of  Bergman-Toeplitz operators in Case I.

\begin{thm}\label{M1}
Let $\varphi$ be a bounded positive function on $\mathbb D$. If $d\mu=\varphi dA$, then the following are equivalent:\\
$(a)$ $T_\mu=T_\varphi$ is invertible on $L_a^2$;\\
$(b)$ $\widetilde{\mu}=\widetilde{\varphi}$ is invertible in $L^\infty(\mathbb D, dA)$, i.e., $\widetilde{\mu}$ is bounded below on $\mathbb D $;\\
$(c)$ There exists an $r>0$ such that  $\chi_{\{z\in \mathbb D: \varphi(z)>r\}} dA$  is a reverse Carleson measure for $L_a^2$;\\
$(d)$ $\mu$ is a  reverse Carleson measure for $L_a^2$;\\
$(e)$ $\mu$ satisfies the   reverse Carleson condition \emph{(\ref{RCC})}.
\end{thm}
\begin{proof}
	$(a) \Leftrightarrow (d)$ and $(b) \Leftrightarrow (e)$ are special cases of Lemmas \ref{M} and \ref{CRC}. $(a) \Leftrightarrow (c)$ was established in \cite[Corollary 3]{Lue}. Lastly, $(a) \Leftrightarrow (b)$ was proven in \cite[Theorem 3.2]{ZZ}.
\end{proof}

 Next,  we turn to the study of the invertibility of Bergman-Toeplitz operators in Case II. Suppose that $\varphi$ is an analytic self-mapping of $\mathbb D$, i.e., $\varphi$ is an analytic mapping from  $\mathbb D$ to $\mathbb D$. Let $\mu_{\varphi}$ be the pull-back measure of the normalized area measure on $\mathbb D$ under the mapping $\varphi$ defined on Borel subsets of $\mathbb D$ by
$$\mu_\varphi(E)=A(\varphi^{-1}(E)).$$
In the following,  we will characterize the invertibility of $T_{\mu_\varphi}$ in terms of the Berezin transform $\widetilde{\mu_{\varphi}}$ and the composition operator $C_\varphi$ (which is defined by $C_\varphi f=f\circ \varphi$) on the Bergman space $L_a^2$.

Recalling that Littlewood's subordination theorem tells us that $C_\varphi$ is bounded on the Bergman space, we have
$$\widetilde{\mu_\varphi}(z)=\int_{\mathbb D}|k_z(w)|^2d\mu_\varphi(w)=\int_{\mathbb D}|k_z(\varphi(w))|^2dA(w)=\|C_\varphi k_z\|^2_{L_a^2}\leqslant \|C_\varphi\|^2.$$
 It follows that $\widetilde{\mu_\varphi}$ is bounded on $\mathbb D$. According to Lemma \ref{BC}, we have that $\mu_\varphi$ is a Carleson measure for $L_a^2$ and the Toeplitz operator  $T_{\mu_\varphi}$ is bounded. Furthermore, the following relationship between  $C_\varphi$ and $T_{\mu_\varphi}$ was obtained in  \cite{Zhu} (see the identity in (11.5) of Section 11.3):
$$T_{\mu_\varphi}f=C_{\varphi}^*C_\varphi f, \ \ \  \ f\in L_a^2.$$

The following lemma characterizes when the composition operator $C_\varphi$ has closed range on the Bergman space, which was established in \cite[Theorems 4.2 and 4.4]{GT}.

\begin{lem}\label{Comp}
Let $\varphi$ be a non-constant analytic self-mapping of the unit disk $\mathbb D$. Then the following statements are equivalent:\\
$(a)$ $\mu_{\varphi}$ satisfies the reverse Carleson condition;\\
$(b)$ $\widetilde{\mu_\varphi}$ is bounded below on $\mathbb D$;\\
$(c)$ $C_\varphi$ has closed range on the Bergman space $L_a^2$.
\end{lem}

Using Lemma \ref{Comp}, we are able to establish several necessary and sufficient conditions concerning the invertibility of the Bergman-Toeplitz operator $T_{\mu_\varphi}$
via the reverse Carleson measure for $L_a^2$ and the Berezin transform of $\mu_\varphi$.

\begin{thm}\label{M2}
Let $\varphi$ be a non-constant analytic self-mapping of $\mathbb D$ and $d\mu_{\varphi}=dA\circ \varphi^{-1}$.  Then the following five conditions are equivalent:\\
$(a)$  $T_{\mu_\varphi}$  is invertible on $L_a^2$;\\
$(b)$ $\widetilde{\mu_\varphi}$  is bounded below on $\mathbb D$;\\
$(c)$ $\mu_\varphi$ satisfies the reverse Carleson condition \emph{(\ref{RCC})};\\
$(d)$  $C_\varphi$ has closed range on $L_a^2$;\\
$(e)$ $\mu_\varphi$ is a reverse Carleson measure for $L_a^2$.
\end{thm}

\begin{proof}
Note that $(a)\Rightarrow (b)$ is obvious and $(b)\Leftrightarrow (c)\Leftrightarrow (d)$ were established in Lemma \ref{Comp}.
In addition, $(a)\Leftrightarrow (e)$ was obtained in Lemma  \ref{M}. Therefore, it is sufficient to show that $(d)\Rightarrow (a)$.

Suppose that $(d)$ holds, i.e., the composition operator $C_\varphi$ has closed range on the Bergman space. Since $C_\varphi$ is injective on $L_a^2$, we conclude by the inverse mapping theorem  that $C_\varphi$ is bounded below on $L_a^2$. Thus there exists a constant $\delta>0$ such that
$$\|C_\varphi f\|_{L_a^2}\geqslant \delta \|f\|_{L_a^2}, \ \ \ \ f\in L_a^2.$$
It follows that
$$\langle C_\varphi^* C_\varphi f, f\rangle_{L_a^2}=\|C_\varphi f\|_{L_a^2}^2\geqslant \delta^2\|f\|_{L_a^2}^2=\delta^2\langle f, f\rangle_{L_a^2}$$
for all $f\in L_a^2$. Thus we have  $$C_\varphi^*C_\varphi \geqslant \delta^2I,$$ which yields that $C_\varphi^* C_\varphi$ is invertible on $L_a^2$. Hence $T_{\mu_{\varphi}}$ is also invertible on the Bergman space, completing the proof of Theorem \ref{M2}.
\end{proof}

 For Toeplitz operators in Cases I and II, the key point of the proofs of Theorems \ref{M1} and \ref{M2} is to show that $T_{\mu}$ is invertible if $\widetilde{\mu}$ is, which relies on the specific expression of $\mu$, i.e., $d \mu = \varphi d A$ or $d \mu_{\varphi} = d A \circ \varphi^ {-1}$.

In the rest of this section, we are going to give a negative answer to the Douglas question for Bergman-Toeplitz operators induced by positive measures (Question \ref{Q2}). Indeed, we will show in the end of this section that the invertibility of a Bergman-Toeplitz operator with a positive measure $\mu$ as the symbol is not completely determined by the invertibility of the Berezin transform of $\mu$.  According to Lemmas \ref{M} and \ref{CRC}, it is sufficient to construct a finite positive  measure  $\mu$ on the open unit disk $\mathbb D$ such that \emph{$\mu$ is not a reverse Carleson measure for the Bergman space but satisfies the reverse Carleson condition (\ref{RCC})}.

Because the invertibility of a Toeplitz operator and its Berezin transform is equivalent in Case I (see Theorem \ref{M1}),  to solve the Douglas question for Bergman-Toeplitz operators induced by general positive measures we try  to consider the measures which are singular with respect to $d A$. This inspires the following observation.

\begin{prop} \label{proposition 1}
 Suppose that $r \in (0, 1)$ and $\Gamma=\{a_k\}_{k=1}^\infty$ is a uniformly discrete $r$-net. Define
\begin{align}\label{mu}
	\mu = \sum_{k = 1}^ {\infty} {(1 - |a_k|^ 2)^ 2 \delta_{a_k}},
\end{align}
where $\delta_{a_k}$ is the point mass concentrated at  $a_k$.  Then $\mu$ is a finite positive Borel measure satisfying the reverse Carleson condition for $L_a^2$.
\end{prop}

\begin{proof}
Clearly, $\mu$ is a finite Borel measure on $\mathbb D$ and $$\mu (E(z, r)) \lesssim  A (E(z, r))$$ for all $z\in \mathbb D$. Let $z$ be any point in $\mathbb{D}$. Since $\{ E(a_k, r) \}_{k = 1}^ {\infty}$ is a covering of $\mathbb{D}$, $z$ is in some $E(a_{k_1}, r)$. It follows that
\begin{align*}
	\mu (E(z, r)) \geqslant (1 - |a_{k_1}|^ 2)^ 2 \asymp A(E(a_{k_1}, r)).
\end{align*}
Because $\rho (z, a_{k_1})$ stays bounded, by (\ref{zw}) we further have
\begin{align*}
	\mu (E(z, r)) \gtrsim A(E(z, r)),
\end{align*}
finishing the proof.
\end{proof}

\begin{prop} \label{proposition 2}
	The measure $\mu$ defined by \emph{(\ref{mu})} is a reverse Carleson measure for $L_a^2$ if and only if  $\Gamma$ is a sampling sequence.
\end{prop}

\begin{proof}
	\par Noting that
	\begin{align*}
		\int_{\mathbb{D}} |f(z)|^2 d \mu(z) = \sum_{k = 1}^ {\infty} {(1 - |a_k|^ 2)^ 2 |f(a_k)|^ 2}, \qquad f \in L_a^ 2,
	\end{align*}
	the assertion follows immediately.
\end{proof}

Given an $R > 0$. It follows from \cite[Lemma 4.8]{Zhu} that there is an $R$-lattice $\Gamma_0 = \{ b_k \}_{k = 1}^ {\infty}$. According to Proposition \ref{proposition 1}, the measure
\begin{align}\label{mu0}
	\mu_0 = \sum_{k = 1}^ {\infty} {(1 - |b_k|^ 2)^ 2 \delta_{b_k}}
\end{align}
also satisfies the reverse Carleson condition. Consequently,  we conclude by Proposition \ref{proposition 2} that \emph{$\mu_0$ is a reverse Carleson measure for the Bergman space  $L_a^2$ if and only if $\Gamma_0$ is a sampling sequence}.

Recently, using a characterization of sampling sequence for $L_a^2$ (see \cite[Theorem 5.23]{HKZ}), Zhuo and Ye \cite{ZY} showed that the measure $\mu_0$ defined in (\ref{mu0}) is not a reverse Carleson measure for $L_a^2$ if the lower Seip density (see (5.22) in \cite{HKZ} for the definition) of $\Gamma_0$ is less than or equal to $\frac{1}{2}$, see Theorem 1.1 and Example 1.2 in \cite{ZY} for the details.

Next, we will further  discuss how the values of $R$ influence $\mu_0$ to be a reverse Carleson measure for $L_a^2$ by using the following  criteria on sampling and interpolation sequences for $L_a^2$.

\begin{lem} \label{SI}
	\emph{(\cite[Theorems 10 and 11 in Chapter 6]{DS}) } Suppose that $\{ u_k \}_{k = 1}^ {\infty}$ is a uniformly discrete sequence with separation constant $\delta$. \\
	$(a)$ If $\{ u_k \}_{k = 1}^ {\infty}$ is an $r$-net  with $0<r < \frac{1}{2}$, then $\{ u_k \}_{k = 1}^ {\infty}$ is a sampling sequence; \\
	$(b)$ If $\delta$ satisfies that
	\begin{align*}
		(2 \pi + 1) \frac{\sqrt{1 - \delta}}{(1 - \sqrt{1 - \delta})^ 2} < \frac{1}{2},
	\end{align*}
	then $\{ u_k \}_{k = 1}^ {\infty}$ is an interpolation sequence.
\end{lem}



It follows immediately from  Lemma \ref{SI} that $\mu_0$ is a reverse Carleson measure for $L_a^2$ when $R\rightarrow 0^+$. On the other hand,
note that the separation constant $\delta (\Gamma_0)$ of $\Gamma_0$ is greater than or equal to $\tanh(\frac{R}{2})$.  This means that $\Gamma_0$ is an interpolation sequence  when $R$ is large enough, so $\Gamma_0$ is not a sampling sequence in this case. This yields that $\mu_0$ is not a reverse Carleson measure for $L_a^2$ if $R$ is sufficiently large.  To summarize, we have the following proposition.

\begin{prop}\label{CR}
	Suppose that $R > 0$ and $\Gamma_0 = \{ b_k \}_{k = 1}^ {\infty}$ is a sequence satisfying
	\begin{align*}
		\mathbb D = \bigcup_{k = 1}^ {\infty} {D(b_k, R)}
	\end{align*}
	and
	\begin{align*}
		\beta (b_k, b_j) \geqslant \frac{R}{2}, \ \ \ \ k \neq j.
	\end{align*}
	Let $\mu_0$ be the measure defined by
$$\mu_0 = \sum_{k = 1}^ {\infty} {(1 - |b_k|^ 2)^ 2 \delta_{b_k}}.$$
 Then we have:\\
	$(a)$  $\mu_0$ is a reverse Carleson measure for the Bergman space if  $R$ is  sufficiently close to $0$;\\
	$(b)$  $\mu_0$ is not a reverse Carleson measure for the Bergman space if  $R$ is  large enough.
\end{prop}

Now we are in the position to give a negative answer to Question \ref{Q2}.

\begin{thm}\label{CE}
There exists a finite positive Borel measure $\nu$ on the open unit disk $\mathbb D$ such that the Toeplitz operator $T_{\nu}$ is not invertible on the Bergman space, but the Berezin transform of $\nu$ is bounded below on $\mathbb D$.
\end{thm}

\begin{proof}
Let $\nu$ be the positive measure constructed  in (b) of Proposition  \ref{CR}. Then we have by Proposition \ref{proposition 1} that $\nu$ satisfies the reverse Carleson condition. It follows from Lemma \ref{CRC} that $\widetilde{\nu}$ is bounded below on the unit disk $\mathbb D$.

However, Lemma \ref{M} tells us that the Toeplitz operator $T_\nu$ is not invertible on the Bergman space $L_a^2$, since $\nu$ is not a reverse Carleson measure for $L_a^2$.
\end{proof}

\section{The Douglas question on the Fock space}\label{DF}

In the final section, we show that the analogues of Lemma \ref{M}, Theorems \ref{M1} and \ref{CE} also hold in the  Fock-space setting.
Let $\mu$ be a positive Borel measure on the complex plane $\mathbb C$ satisfying condition (M) and $T_\mu$ be the Toeplitz operator induced by $\mu$ on the Fock space $\mathcal F^2$. Let us first review some characterizations of the Carleson measure for the Fock space.

\begin{lem}\label{FockC}
\emph{(\cite[Theorme 3.29]{Zhu2})} Suppose that $\mu$ is positive Borel measure on $\mathbb C$. Then the following statements are equivalent:\\
\emph{($a$)} $\mu$ is a Carleson measure for $\mathcal F^2$;\\
\emph{($b$)} There exists a constant $C > 0$ such that
	\begin{align*}
		\int_{\mathbb{C}} |f(w)| \mathrm{e}^ {-\frac{|w|^ 2}{2}} d \mu (w) \leqslant C \int_{\mathbb{C}} |f(w)| \mathrm{e}^ {-\frac{|w|^ 2}{2}} d A(w)
	\end{align*}
for every  entire function $f$;\\
\emph{($c$)} There exists a constant $C>0$ such that
$$\int_{\mathbb{C}} |k_z^ {(\mathcal{F}^ 2)} (w)|^ 2 \mathrm{e}^ {-\frac{|w|^ 2}{2}} d \mu (w) = \int_{\mathbb C} \mathrm{e}^{-\frac{|z-w|^2}{2}}d\mu(w)\leqslant C$$
for all $z\in \mathbb C$;\\
\emph{($d$)} For every $r > 0$, there exists a constant $C>0$ (depending only on $r$) such that
$$\mu(B(z, r))\leqslant C$$
for all $z\in \mathbb C$.
\end{lem}

The next proposition is parallel to Lemma \ref{BC}.

\begin{prop} \label{BF}
	Suppose $\mu$ is a positive Borel measure on $\mathbb{C}$ satisfying the condition (M). Then the following are equivalent:\\
	\emph{$(a)$} $T_{\mu}$ is bounded on  the Fock space $\mathcal{F}^2$; \\
	\emph{$(b)$} $\widetilde{\mu}$ is bounded on $\mathbb C$;\\
	\emph{$(c)$} $\mu$ is a Carleson measure for  $\mathcal F^2$.
\end{prop}

\begin{proof}
	This  can be proved by using the same technique as the one used in the proof of Lemma \ref{BC}. However, we shall include a detailed  proof here for the sake of completeness.
	
	We first show $(a)\Rightarrow (b)$.  Suppose that $T_\mu$ is bounded on the Fock space, then the Cauchy-Schwarz inequality gives
	$$\widetilde{\mu}(z)=\frac{1}{2}\int_{\mathbb C} |k_z^{(\mathcal F^2)}(w)|^2 \mathrm{e}^ {-\frac{|w|^ 2}{2}} d\mu(w)=\big|\langle T_\mu k_z^{(\mathcal F^2)}, k_z^{(\mathcal F^2)} \rangle_{\mathcal{F}^ 2} \big|\leqslant \|T_\mu\|, \ \ \ \ z\in \mathbb C, $$
	where the second equality comes from the property of  reproducing kernels for $\mathcal F^2$.
	
	Note that $(b)\Rightarrow (c)$ is a direct consequence of  ``$(c)\Rightarrow (a)$" in  Lemma \ref{FockC}.  It remains to show that $(c)\Rightarrow (a)$. If $\mu$ is a Carleson measure for $\mathcal F^2$, then for any analytic polynomial $f$  we have  by  Lemma \ref{FockC} that
	\begin{align*}
		|T_{\mu} f (z)| & \leqslant  \frac{1}{2} \int_{\mathbb{C}} |f(w)|~|K^{(\mathcal F^2)}_w(z)| \mathrm{e}^ {-\frac{|w|^ 2}{2}} d \mu (w) \\
		& \lesssim  \frac{1}{2} \int_{\mathbb{C}} |f(w) K^{(\mathcal F^2)}_w(z)| \mathrm{e}^ {-\frac{|w|^ 2}{2}} d A(w) \\
		& =  \int_{\mathbb{C}} |f(w) K^{(\mathcal F^2)}_w(z)| d \lambda (w) \\
		& =  S (|f|) (z), \ \ \ \ z\in \mathbb C,
	\end{align*}
	where  $S$ is the linear operator defined by
	\begin{align}\label{opS}
		S h (z) = \int_{\mathbb{C}} h(w) |K_z^ {(\mathcal{F}^ 2)} (w)|  d \lambda (w), \ \ \ \ h \in L^ 2 (\mathbb{C}, d\lambda).
	\end{align}
According to \cite[Theorem 2.20]{Zhu2}, the operator $S$ is bounded on $L^ 2 (\mathbb{C}, d \lambda)$. Then the boundedness of $T_{\mu}$  follows from the boundedness of $S$, to complete the proof.
\end{proof}

The following identity is quite useful in the study of the bounded Fock-Toeplitz operators with measures as symbols.

\begin{lem}\label{IntegralF}
	Suppose that $\mu$ is a positive Borel measure satisfying condition  (M) and $T_{\mu}$ is bounded on the Fock space $\mathcal F^2$. Then
	\begin{align*}
		\langle T_{\mu} f, g \rangle_{\mathcal{F}^ 2} = \frac{1}{2} \int_{\mathbb{C}} f (z) \overline{g(z)} \mathrm{e}^ {-\frac{|z|^ 2}{2}} d \mu (z)
	\end{align*}
	for $f$ and  $g$ in $\mathcal{F}^2$.
\end{lem}

\begin{proof}
	Using  the integral expression of $T_{\mu} f$, we have
	\begin{align*}
		\langle T_{\mu} f, g \rangle_{\mathcal{F}^2} = \frac{1}{2} \int_{\mathbb{C}} \int_{\mathbb{C}} f(w) \overline{g(z)} K_w^{(\mathcal F^2)}(z) \mathrm{e}^ {-\frac{|w|^ 2}{2}} d \mu (w) d \lambda (z).
	\end{align*}
	In order to use Fubini's theorem to finish the proof, we need to  show that
	\begin{align*}
		\int_{\mathbb{C}} \int_{\mathbb{C}} |f(w)|~|g(z)|~|K_w^{(\mathcal F^2)}(z)| \mathrm{e}^ {-\frac{|w|^2}{2}} d \mu (w) d \lambda (z)<+\infty.
	\end{align*}
	Indeed, since $T_{\mu}$ is bounded, $\mu$ is a Carleson measure by Proposition \ref{BF}. According to Lemma \ref{FockC}, there exists a positive constant $C$ such that
	\begin{align*}
		\int_{\mathbb{C}} |f(w)|~|K_w^{(\mathcal F^2)}(z)| \mathrm{e}^ {-\frac{|w|^ 2}{2}} d \mu (w) & \leqslant  C \int_{\mathbb{C}} |f(w) K_w^{(\mathcal F^2)}(z)| \mathrm{e}^ {-\frac{|w|^ 2}{2}} d A(w) \\
		& = 2 C \int_{\mathbb{C}} |f(w) K_w^{(\mathcal F^2)}(z)| d \lambda (w) \\
		& = 2 C S (|f|) (z),
	\end{align*}
	where $S$ is the bounded operator defined in (\ref{opS}). Noting that both $g$ and $|f|$ are in $L^ 2 (\mathbb{C}, d \lambda)$, we arrive at
	\begin{align*}
		\int_{\mathbb{C}} \int_{\mathbb{C}} |f(w)|~|g(z)|~|K_w^{(\mathcal F^2)}(z)| \mathrm{e}^ {-\frac{|w|^ 2}{2}} d \mu (w) d \lambda (z) & \leqslant 2 C \int_{\mathbb{C}} |g(z)|~S(|f|) (z) d \lambda (z) \\
		&<+\infty.
	\end{align*}
Applying Fubini's theorem, we obtain that
	\begin{align*}
		\langle T_{\mu} f, g \rangle_{\mathcal{F}^ 2} & = \int_{\mathbb{C}} T_{\mu} f(z) \overline{g(z)} d \lambda (z) \\
		& = \frac{1}{2} \int_{\mathbb{C}} \int_{\mathbb{C}} f(w) \overline{g(z)} K_w^{(\mathcal F^2)}(z) \mathrm{e}^ {-\frac{|w|^ 2}{2}} d \mu (w) d \lambda (z) \\
		& = \frac{1}{2} \int_{\mathbb{C}} f(w) \mathrm{e}^ {-\frac{|w|^ 2}{2}} \int_{\mathbb{C}} \overline{g(z)} K_w^{(\mathcal F^2)}(z) d \lambda (z) d \mu (w) \\
		& = \frac{1}{2} \int_{\mathbb{C}} f(w) \overline{g(w)} \mathrm{e}^ {-\frac{|w|^ 2}{2}} d \mu (w),
	\end{align*}
	as desired.
\end{proof}

Now we can establish the analogue of Lemma \ref{M} on the Fock space by using Lemma \ref{IntegralF} and Proposition \ref{BF}.

\begin{thm}\label{MF}
Assume that $\mu$ is a positive Borel measure on $\mathbb C$ satisfying the condition (M). If $T_{\mu}$ is bounded on the Fock space $\mathcal F^2$,  then $T_{\mu}$ is invertible  on $\mathcal F^2$  if and only if $\mu$ is a reverse Carleson measure for $\mathcal F^2$.
\end{thm}

\begin{proof}
For simplicity, we denote
\begin{align*}
			d \mu_0(w) = \frac{1}{2} \mathrm{e}^ {-\frac{|w|^ 2}{2}} d \mu (w), \ \ \ \  w\in \mathbb C.
\end{align*}
If $T_{\mu}$ is invertible on $\mathcal F^2$, then we can find a positive constant $\delta$  such that $T_{\mu}\geqslant \delta I$. It follows from Lemma  \ref{IntegralF} that
		\begin{align*}
			\delta \| f \|_{\mathcal{F}^ 2}^ 2 \leqslant  \langle T_{\mu} f, f \rangle_{\mathcal{F}^ 2} = \langle f, f \rangle_{L^ 2 (\mathbb{C}, d \mu_0)} = \| f \|_{L^ 2 (\mathbb{C}, d \mu_0)}^ 2, \ \ \ \ \ f \in \mathcal{F}^ 2.
		\end{align*}
This proves the necessity. 		

Conversely, assume that $\mu$ is a reverse Carleson measure for $\mathcal F^2$. Since $T_\mu$ is bounded on $\mathcal F^2$, we have by Proposition \ref{BF} that $\mu$ is also a Carleson measure for $\mathcal F^2$. This yields that
		\begin{align}\label{eq2}
			\| f \|_{\mathcal{F}^ 2} \asymp \| f \|_{L^ 2 (\mathbb{C}, d \mu_0)}, \qquad f \in \mathcal{F}^ 2.
		\end{align}
Furthermore, we have that $\mathcal{F}^2$ is a closed subspace of $L^ 2(\mathbb{C}, d \mu_0)$. Let $Q$ be the orthogonal projection from $L^2 (\mathbb{C}, d \mu_0)$ onto $\mathcal{F}^2$. Then (\ref{eq2}) gives
		\begin{align*}
			\| f \|_{\mathcal{F}^ 2} \asymp \sup \left\{ |\langle f, g \rangle_{L^2 (\mathbb{C}, d \mu_0)} | : g \in L^2 (\mathbb{C}, d \mu_0) \ \ \ \mathrm{and}\ \  \ \| g \|_{L^2 (\mathbb{C}, d \mu_0)} \leqslant  1 \right\}
		\end{align*}
for any $f\in\mathcal{F}^2$.

Let  $g$  be in $L^ 2 (\mathbb{C}, d \mu_0)$ with $\| g \|_{L^ 2 (\mathbb{C}, d \mu_0)} \leqslant  1$ and $h=Qg$. It follows from Lemma \ref{IntegralF} and (\ref{eq2}) that
		\begin{align*}
			|\langle f, g \rangle_{L^ 2 (\mathbb{C}, d \mu_0)} | & = |\langle f, h \rangle_{L^ 2 (\mathbb{C}, d \mu_0)}| = |\langle T_{\mu} f, h \rangle_{\mathcal{F}^ 2}| \\
			& \leqslant \| T_{\mu} f \|_{\mathcal{F}^ 2} \| h \|_{\mathcal{F}^ 2} \asymp \| T_{\mu} f \|_{\mathcal{F}^ 2} \| h \|_{L^ 2 (\mathbb{C}, d \mu_0)} \\
			& \leqslant  \| T_{\mu} f \|_{\mathcal{F}^ 2}.
		\end{align*}
		Therefore
		\begin{align*}
			\| f \|_{\mathcal{F}^ 2} \lesssim \| T_{\mu} f \|_{\mathcal{F}^ 2}, \ \ \ \  f \in \mathcal{F}^ 2,
		\end{align*}
which implies that the positive Toeplitz operator $T_{\mu}$ is invertible on $\mathcal F^2$.
\end{proof}

The next proposition shows that a Carleson measure for the Fock space has the nice property similar to the conclusion in Lemma \ref{Lim} on the Bergman space.
\begin{prop} \label{LimF}
If $\mu$ is a Carleson measure for the Fock space $\mathcal F^2$, then
\begin{align*}
	\lim_{R \rightarrow +\infty} {\sup_{z \in \mathbb{C}} {\int_{\mathbb{C} \setminus B(z, R)} \mathrm{e}^ {-\frac{|z - w|^ 2}{2}} d \mu (w)}} = 0.
\end{align*}
\end{prop}

\begin{proof}
	 Fix $R> 2 \sqrt{2}$. For $m, n\in \mathbb Z$, we define
\begin{align*}
			Q_{m, n} = \left\{ z \in \mathbb{C} : m \leqslant \mathrm{Re}(z) \leqslant  m + 1, \ n \leqslant  \mathrm{Im}(z)\leqslant n + 1 \right\}.
		\end{align*}
		Denote $\{Q_{m, n}: m, n\in \mathbb Z\}$ by  $\{ Q_k \}_{k = 1}^ {\infty}$ and define
		\begin{align*}
			I = I (z, R) = \left\{ k \geqslant 1: Q_k \nsubseteq B(z, R) \right\}.
		\end{align*}
		It follows that
		\begin{align*}
			\int_{\mathbb{C} \setminus B(z, R)} \mathrm{e}^ {-\frac{|z - w|^ 2}{2}} d \mu (w) = \sum_{k \in I} {\int_{Q_k \setminus B(z, R)} \mathrm{e}^ {-\frac{|z - w|^ 2}{2}} d \mu (w)} \leqslant \sum_{k \in I} {\int_{Q_k} \mathrm{e}^ {-\frac{|z - w|^ 2}{2}} d \mu (w)}.
		\end{align*}
		
Let $H(z, \alpha, \beta)$ with $0 < \alpha < \beta$ be the annulus
		\begin{align*}
			\left\{ w \in \mathbb{C} : \alpha \leqslant  |w-z|\leqslant \beta \right\}.
		\end{align*}
		If we set
		\begin{align*}
			I_j = I_j (z, R) = \left\{ k \geqslant 1: Q_k \cap H(z, R + j, R + j + 1) \neq \varnothing \right\}, \ \ \  j = 0, 1, 2, \cdots,
		\end{align*}
		then clearly $I = \bigcup\limits_{j = 0}^ {\infty} {I_j}$ and
		\begin{align*}
			\int_{\mathbb{C} \setminus B(z, R)} \mathrm{e}^ {-\frac{|z - w|^ 2}{2}} d \mu (w) \leqslant \sum_{j = 0}^ {\infty} {\sum_{k \in I_j} {\int_{Q_k} \mathrm{e}^ {-\frac{|z - w|^ 2}{2}} d \mu (w)}}.
		\end{align*}
	
For each $k \in I_j$, since $Q_k \cap H(z, R + j, R + j + 1) \neq \varnothing$, by the triangle inequality we have
		\begin{align}
			R + j - \sqrt{2} \leqslant |w - z| \leqslant  R + j + 1 + \sqrt{2}, \qquad w \in Q_k. \label{ineq7}
		\end{align}
This implies that
		\begin{align*}
			Q_k \subseteq H(z, R + j - \sqrt{2}, R + j + 1 + \sqrt{2})
		\end{align*}
when $k \in I_j$, which yields that
		\begin{align}
			\mathrm{card} (I_j) & \leqslant  A\big(H(z, R + j - \sqrt{2}, R + j + 1 + \sqrt{2})\big) \notag \\
			& \leqslant \pi (1 + 2 \sqrt{2}) (2 R + 2 j + 1) < 60 (R + j). \label{ineq8}
		\end{align}
		Moreover, we obtain by (\ref{ineq7}) that
		\begin{align*}
			\int_{Q_k} \mathrm{e}^ {-\frac{|z - w|^ 2}{2}} d \mu (w) \leqslant \mathrm{e}^ {-\frac{(R + j - \sqrt{2})^ 2}{2}} \mu (Q_k),\ \ \ \  \ k\in I_j.
		\end{align*}
		Since each cube $Q_k$ is contained in some disk $B (u, 1)$ and  $\mu$ is a Carleson measure for $\mathcal F^2$, Lemma \ref{FockC} tells us that there is a constant $C>0$ such that
		\begin{align*}
			\mu (Q_k)\leqslant  \mu (B(u, 1)) \leqslant C.
		\end{align*}
		Thus we have
		\begin{align*}
			\int_{Q_k} \mathrm{e}^ {-\frac{|z - w|^ 2}{2}} d \mu (w) \leqslant C \mathrm{e}^ {-\frac{(R + j - \sqrt{2})^ 2}{2}},\ \ \ \  \ k\in I_j.
		\end{align*}
		It follows from (\ref{ineq8}) that
		\begin{align*}
			\sum_{k \in I_j} {\int_{Q_k} \mathrm{e}^ {-\frac{|z - w|^ 2}{2}} d \mu (w)} & \leqslant  C \sum_{k \in I_j} {\mathrm{e}^ {-\frac{(R + j - \sqrt{2})^ 2}{2}}}\\
			& \leqslant 60 C (R + j) \mathrm{e}^ {-\frac{(R + j - \sqrt{2})^ 2}{2}}.
		\end{align*}
Now we arrive at
		\begin{align*}
			\int_{\mathbb{C} \setminus B(z, R)} \mathrm{e}^ {-\frac{|z - w|^ 2}{2}} d \mu (w) & \leqslant 60 C \sum_{j = 0}^ {\infty} {(R + j) \mathrm{e}^ {-\frac{(R + j - \sqrt{2})^ 2}{2}}} \\
			& \leqslant 60 C \sum_{j = 0}^ {\infty} {(R + j) \mathrm{e}^ {-\frac{(R + j)^ 2}{8}}} \\
			& \leqslant 60 C \sum_{j = 0}^ {\infty} {(R + j) \mathrm{e}^ {-\frac{R^ 2 + j^ 2}{8}}} \\
			& \lesssim \mathrm{e}^ {-\frac{R^ 2}{8}} \left(  R \sum_{j = 0}^ {\infty} {\mathrm{e}^ {-\frac{j^ 2}{8}}} + \sum_{j = 0}^ {\infty} {j \mathrm{e}^ {-\frac{j^ 2}{8}}} \right),
		\end{align*}
where the second inequality is due to $R>2\sqrt{2}$. This completes the proof of Proposition \ref{LimF}.
	\end{proof}

According to Proposition \ref{LimF}, we obtain the following characterization of a Carleson measure for the Fock space that satisfies the  reverse Carleson condition, which is parallel to Lemma \ref{CRC} in the setting of the Bergman space.

\begin{thm}\label{CRCF}
 A Carleson measure $\mu$  for the Fock space $\mathcal F^2$ satisfies the  reverse Carleson condition if and only if the Berezin transform $\widetilde{\mu}$ is bounded above and below on $\mathbb C$.
\end{thm}

\begin{proof}
Assume  that $\mu$ satisfies the reverse Carleson condition for the Fock space. Then there is an $r>0$  and a constant $c>0$ such that $\mu(B(z, r))\geqslant c$ for $z\in \mathbb C$. It follows that
\begin{align*}
			\widetilde{\mu} (z) & = \frac{1}{2} \int_{\mathbb{C}} \mathrm{e}^ {-\frac{|z - w|^ 2}{2}} d \mu (w)
			 \geqslant \frac{1}{2} \int_{B(z, r)} \mathrm{e}^ {-\frac{|z - w|^ 2}{2}} d \mu (w) \\
			& \geqslant  \frac{\mathrm{e}^ {-\frac{r^ 2}{2}}}{2} \mu (B(z, r))
			 \geqslant \frac{c \mathrm{e}^ {-\frac{r^ 2}{2}}}{2}
\end{align*}
for all $z\in \mathbb{C}$, to get that $\widetilde{\mu}$ is bounded below on $\mathbb C$.
		
Conversely, suppose that $\widetilde{\mu}$ is bounded below, i.e.,  there is a constant $\varepsilon_0 > 0$ such that
		\begin{align*}
			\varepsilon_0 \leqslant  \widetilde{\mu} (z), \qquad z \in \mathbb{C}.
		\end{align*}
Then it is sufficient to show that there is a $\rho_0>0$ and a constant $C>0$ such that  $\mu({B(z, \rho_0)})\geqslant C$ for all $z\in \mathbb C$. To
this end, 	
 we write
		\begin{align*}
			\widetilde{\mu} (z) = \frac{1}{2} \int_{B(z, \rho)} \mathrm{e}^ {-\frac{|z - w|^ 2}{2}} d \mu (w) + \frac{1}{2} \int_{\mathbb{C} \setminus B(z, \rho)} \mathrm{e}^ {-\frac{|z - w|^ 2}{2}} d \mu (w)
		\end{align*}
for  $\rho>0$.  By Proposition \ref{LimF}, there exists some  $\rho_0$ large enough such that
		\begin{align*}
			\frac{1}{2} \int_{\mathbb{C} \setminus B(z, \rho_0)} \mathrm{e}^ {-\frac{|z - w|^ 2}{2}} d \mu (w) < \frac{\varepsilon_0}{2}, \qquad z \in \mathbb{C}.
		\end{align*}
		From this we deduce that
		\begin{align*}
			\frac{\varepsilon_0}{2} \leqslant  \frac{1}{2} \int_{B(z, \rho_0)} \mathrm{e}^ {-\frac{|z - w|^ 2}{2}} d \mu (w) \leqslant  \frac{\mu (B(z, \rho_0))}{2}
		\end{align*}
for all $ z \in \mathbb{C}$, which implies that $\mu$ satisfies the reverse Carleson condition for $\mathcal F^2$.
	\end{proof}

Combining Theorem \ref{MF}, Theorem \ref{CRCF} and  the conclusions in \cite[Theorem 1.1]{WZ}, we obtain the following  necessary and sufficient conditions for the invertibility of the Toeplitz operator $T_\varphi$ on the Fock space, where $\varphi$ is a bounded  positive function on the complex plane $\mathbb C$.

\begin{thm}\label{FM1}
Let $\varphi$ be a bounded positive function on $\mathbb C$. If $d\mu=\varphi dA$, then the following statements are equivalent:\\
$(a)$ $T_\mu=T_\varphi$ is invertible on $\mathcal F^2$;\\
$(b)$ $\widetilde{\mu}=\widetilde{\varphi}$ is bounded below on $\mathbb  C $;\\
$(c)$ There exists an $r>0$ such that  $\chi_{\{z\in \mathbb D: \varphi(z)>r\}} d\lambda$  is a reverse Carleson measure for $\mathcal F^2$;\\
$(d)$ $\mu$ is a  reverse Carleson measure for $\mathcal F^2$;\\
$(e)$ $\mu$ satisfies the  reverse Carleson condition for $\mathcal F^2$.
\end{thm}

In view of the discussions for Bergman-Toeplitz operators in the previous section, it is  not difficult to guess that Theorem \ref{FM1} does not hold for Toeplitz operators induced by general positive measures on the Fock space.  In fact, using the conclusions in Theorems \ref{MF} and \ref{CRCF} and following the idea of the construction in Proposition \ref{proposition 1}, it is natural for us to consider the measure of the form:
\begin{align*}
	\mu = \sum_{k = 1}^ {\infty} {\delta_{a_k}},
\end{align*}
where $\delta_{a_k}$ is the point mass concentrated at $a_k$. Furthermore, one can extend Propositions \ref{proposition 1} and \ref{proposition 2}
to the Fock-space setting without many  difficulties.

 In order to answer  Question \ref{Q3}, it suffices to consider the measure defined above with $\{ a_k \}_{k = 1}^ {\infty}$  an $r$-lattice for $\mathcal{F}^ 2$, which was investigated by Zhuo and Lou in Section 1 of \cite{ZL}. Let us quote their result as the following lemma.

\begin{lem}\label{muF}
Let $\{a_k\}_{k=1}^\infty$ be the sequence given in \emph{(\ref{Lambda})}, that is,
$$\Lambda(0, r, \mathrm{i}r)=\{a_k\}_{k=1}^\infty.$$		
Define the positive Borel measure
\begin{align*}
		\mu = \sum_{k = 1}^ {\infty} {\delta_{a_k}},
\end{align*}
where $\delta_{a_k}$ is the point mass concentrated at $a_k$.
Then $\mu$ satisfies the reverse Carleson condition for the Fock space $\mathcal F^2$, but $\mu$ is not a reverse Carleson measure for $\mathcal F^2$ provided $r\geqslant \sqrt{2\pi}$.
\end{lem}

Now we are ready to give a negative answer to Question \ref{Q3}. In fact, combining Theorem \ref{MF}, Theorem \ref{CRCF} and Lemma \ref{muF}, we obtain the following analogue of Theorem \ref{CE} on the Fock space.

\begin{thm}\label{TF}
Let $r\geqslant \sqrt{2\pi}$ and $\Lambda(0, r, \mathrm{i} r)=\{a_k\}_{k=1}^\infty$. Then $\nu=\sum\limits_{k=1}^\infty \delta_{a_k}$
is not a reverse Carleson measure  for the Fock space $\mathcal F^2$ but satisfies the condition (M) and the reverse Carleson condition for $\mathcal F^2$. Consequently,
the Berezin transform of $\nu$ is bounded below on $\mathbb C$ but the corresponding  Toeplitz operator $T_\nu$ is not invertible on the Fock space $\mathcal F^2$.
\end{thm}

\begin{proof}
Using Lemma \ref{muF}, we have that $\nu$ satisfies the reverse Carleson condition for the Fock space $\mathcal F^2$, so $\nu$ is a Carleson measure for $\mathcal F^2$.  It follows from Lemma \ref{FockC} that
\begin{align*}
	\int_{\mathbb{C}} |K_z^ {(\mathcal{F}^ 2)} (w)|^ 2 \mathrm{e}^ {-\frac{|w|^ 2}{2}} d \nu (w) = \| K_z^ {(\mathcal{F}^ 2)} \|_{\mathcal{F}^ 2}^ 2 \int_{\mathbb{C}} |k_z^ {(\mathcal{F}^ 2)} (w)|^ 2 \mathrm{e}^ {-\frac{|w|^ 2}{2}} d \nu (w) < +\infty
\end{align*}
for each $z \in \mathbb{C}$, implying that $\nu$ satisfies condition (M).

By Theorems \ref{MF} and \ref{CRCF}, we easily deduce that $\widetilde{\nu}$ is bounded below on the complex plane but $T_{\nu}$ is not invertible
on the Fock space $\mathcal F^2$. This completes the proof of Theorem \ref{TF}.
\end{proof}

\subsection*{Acknowledgment} The first author was supported by NSFC (grant number: 12161141013) and Hunan Provincial Natural Science Foundation of China (grant number: 2022JJ30233). The third author was supported by  NSFC (grant number: 12371125) and the Fundamental Research Funds for the Central Universities (grant number: 2022CDJJCLK002).

\end{document}